\documentclass[10pt,twoside]{amsart}
\usepackage[all]{xy}
        \usepackage {amssymb,latexsym,amsthm,amsmath,mathtools,mathrsfs}
        \usepackage{enumitem,color}
        
\usepackage{array}

\allowdisplaybreaks[4]
\usepackage{hyperref}
\usepackage[capitalize,nameinlink]{cleveref} 
\usepackage{xcolor} 
\hypersetup{
    colorlinks,
    linkcolor={red!80!black},
    citecolor={green!80!black},
    urlcolor={blue!80!black}
}
\usepackage[maxbibnames=99]{biblatex} 
\usepackage{mathtools}

\long\def\symbolfootnote[#1]#2{\begingroup%
\def\thefootnote{\fnsymbol{footnote}}\footnote[#1]{#2}\endgroup}

\newcommand{\Z}{\mathbb{Z}}
\newcommand{\C}{\mathscr{Z}}

\newcommand{\diag}{\textup{diag}}
\newcommand{\GL}{\mathrm{GL}}

\newcommand{\Sp}{\mathrm{Sp}}

\newcommand{\fqn}[1]{\mathbb{F}_{q^{#1}}}

\newcommand{\fq}{\ensuremath{\mathbb{F}_q}}
\newcommand{\On}[1]{\mathrm{O}^{#1}}

\newcommand{\U}{\mathrm{U}}

\newcommand{\R}{\mathrm{R}}
\makeatletter
\def\imod#1{\allowbreak\mkern10mu({\operator@font mod}\,\,#1)}
\makeatother

\newtheorem{theorem}{Theorem}[section]
\newtheorem{lemma}[theorem]{Lemma}
\newtheorem{corollary}[theorem]{Corollary}
\newtheorem{proposition}[theorem]{Proposition}
\newtheorem*{theorem*}{Theorem}
\theoremstyle{definition}

\numberwithin{equation}{section}

\newcommand{\ignore}[1]{}

\newcommand{\mynote}[1]{}

\newcommand{\G}{\mathrm G}

\usepackage{parskip}
\begin{document}
\setcounter{section}{0}
\title[Fibers of the square map]{Fibers of the square map in some finite groups \\of Lie type and an application}
\author{Saikat Panja}
\email{panjasaikat300@gmail.com}
\address{Harish-Chandra Research Institute- Main Building, Chhatnag Rd, Jhusi, Uttar Pradesh 211019, India}
\thanks{The author is supported by a PDF-Math fellowship from the Harish-Chandra Research Institute.}
\date{\today}
\subjclass[2020]{20G40, 20D06}
\keywords{word maps, finite groups of Lie type, square root, real conjugacy classes}
\begin{abstract}
    Let $G$ be one of the finite general linear, unitary, symplectic or orthogonal groups over finite fields of odd order. We find the cardinality of the fibers of the square map at a given arbitrary element of $G$. Using this we find the number of real conjugacy classes of $G$. This is primarily achieved by leveraging a recent solution to Brauer's problem 14.
\end{abstract}
\maketitle
\section{Introduction}
In recent decades, there has been increasing interest and progress in the theory of word maps in groups, driven by various motivations and applications.
By a \emph{word}, we refer to an element $\omega$ in the free group on $t$ generators, denoted $\mathscr{F}_t$.
For any group $G$ and a word $\omega$, there exists a map $\widetilde{\omega}: G^t \longrightarrow G$ through evaluation, which is known as \emph{word map}.
Some of the important questions studied extensively are: 
(a) analyzing the image of the word map,
(b) determining the width of the word, defined as the smallest $\ell > 0$ such that $\omega(G)^\ell = \langle \omega(G) \rangle$,
(c) investigating the fibers of the map,
(d) identifying the kernel of the map, i.e., the inverse image of the identity under the word map. 
A particularly intriguing result emerges for the second question when considering finite nonabelian simple groups. Refer to \cite{Shalev09WordMaps} and \cite{LarsenShalevTiep2011WaringSimple} for further details, where the width of non-trivial words in nonabelian finite simple groups of large orders was ultimately demonstrated to be two by Larsen, Shalev, and Tiep.
Refer to the following articles, namely \cite{Segal09VerbalWidth}, \cite{LarsenShalev12Application}, \cite{Guralnick18Surjective}, and \cite{Shalev13Survey}, to explore the nature of results and the state of available methodologies in the area of word maps.
Even more captivating results exist for certain specific words. 
For instance, if $\omega$ represents the commutator word $[x_1,x_2]$, it turns out that $\omega(G)=G$ for a nonabelian finite simple group $G$, meaning that every element of a nonabelian finite simple group is a commutator. 
This significant achievement was established by Liebeck, O’Brien, Shalev, and Tiep \cite{LiebeckObrienShalevTiep10Ore}, thereby resolving a long-standing conjecture posed by Ore in 1951 \cite{Ore51commutator}.
It is worth noting that, generally, the width of a word cannot be $1$. This observation comes from examining power maps, which take the form $x \mapsto x^M$ for some $M \geq 2$.
The focus of our investigation in this article is the square map associated with $M=2$, applied to certain finite groups of Lie type, namely the general linear group, unitary group, symplectic group, and orthogonal group over finite fields of odd order. 
The analysis of the image of the power maps for these groups has been explored in the following articles: \cite{KunduSingh22} (for the general linear group), \cite{PanjaSingh22symporth} (for the symplectic and orthogonal groups), and \cite{PanjaSingh2023unitary} (for the unitary groups). The findings in these works are articulated in terms of generating functions and categorize the conjugacy classes (and elements) of distinct types, namely separable, cyclic, regular, and semisimple.
In a recent study, we determined the kernel of the general power maps within these groups, see \cite{Panja2023roots}.
The main goal of this article is to determine the fiber of the square map at any element of the given group.
We revisit the groups introduced in \cref{sec:preliminaries}. In \cref{sec:Lie-type}, we identify the elements possessing a square root. In \cref{pro:number-of-square-root-arbitrary-GL}, we present the result for the general linear group, specifying the count of square roots for an arbitrary element. This result can be readily extended to the other groups discussed in the article. In \cref{sec:app}, we employ \cref{pro:number-of-square-root-arbitrary-GL} to determine the number of real conjugacy classes of the finite general linear group.
\section{Preliminaries}\label{sec:preliminaries}
While the groups discussed in this article are widely recognized, we briefly introduce them to establish our notation.
We work with a finite field of odd order $q$, where $q$ is a power of a prime. 
By $\GL_n(q)$ we denote the group of all invertible matrices with entries in the finite field $\fq$.
The subgroup of $\GL_n(q^2)$, fixing a non-degenerate Hermitian form is called a unitary group and will be denoted by $\U_n(q^2)$. 
Since all non-degenerate Hermitian forms on a $\fqn{2} ^n$ are unitarily congruent to the standard form, one concludes that all the unitary groups arising are conjugate with each other inside $\GL_n(q^2 )$.
The symplectic group is the subgroup of $\GL_{2n}(q )$ consisting of those elements which preserve a non-degenerate alternating bilinear form on $\fq ^{2n}$. 
The symplectic group will be denoted by $\Sp_{2n}(q)$. 
Lastly, the orthogonal group consists of elements of $\GL_n(q )$ which preserve a non-degenerate quadratic form $Q$ on $\fq ^n$. 
There are two non-isomorphic orthogonal groups when $n$ is even and one orthogonal group up to isomorphism when $n$ is odd. 
As is customary, we denote all them by $\On{\epsilon} _n(q)$ for $\epsilon\in\{\pm,0\}$.
We will require the parametrization of conjugacy classes and the determination of centralizer sizes for arbitrary elements in all the groups mentioned above.
These are widely known results, and we will refrain from restating them here. 
Nonetheless, we provide references for interested readers.
The conjugacy classes of ${\GL_n}(q)$ are established in \cite{Macdonald81conjugacy} by Macdonald (also refer to \cite{Green55Characters} by Green).
For ${\U_n}(q^2)$, the conjugacy classes are detailed in \cite{Ennola62} by Ennola and \cite{Wall63} by Wall.
The conjugacy classes for ${\Sp}_{2n}(q)$ and $\On{\epsilon} _n(q)$ are determined in \cite{Wall63} by Wall. 
They also mention the centralizer sizes in the respective groups.
To determine the classes which have a square root we need a result of Butler.
The result for $\GL_n(q)$ was determined in \cite{KunduSingh22} in a more general context, although we won't need them. 
We will reference the result in \cref{sec:Lie-type} for convenient access.
For an integer $s$ coprime to $q$, let $M(s;q)$ denote the multiplicative order of $q$ in $\Z/s\Z^\times$. Then we have the following result;
\begin{lemma}\cite[Theorem]{Butler55}\label{lem:factorization-of-power}
    Suppose that the polynomial $f(x)$ of degree $d$ is irreducible over $\fq$ and $t$ is the order of its roots. If $m$ is a natural number such that
    $(m,q)=1$, let $m=m_1m_2$ satisfying $(m_1,t)=1$, and each prime factor of
    $m_2$ is a divisor of $t$. Then
    \begin{enumerate}
        \item each root of $f(x^m)$ has order of the form $rm_2t$ where $r$ is a divisor of $m_1$; and
        \item if $e|m_1$, then $f(x^m)$ has exactly $dm_2\phi(e)/M(em_2t;q)$
        irreducible factors of degree $M(em_2t;q)$ over $\fq$, with roots of order $em_2t$.
    \end{enumerate}
\end{lemma}
We also require the findings detailed in \cite{Panja2023roots}, which determined the count of roots of unity of a specific order through the utilization of generating functions. 
This information is essential for establishing the conclusive outcome outlined in the article. 
Please refer to \cref{thm:number-of-real-irreducible-characters} for the final result. 
We mention a typical result from \cite{Panja2023roots}, which we need to state \cref{thm:number-of-real-irreducible-characters}. 
Note that the statement is in terms of probability generating function.
The number of certain roots of identity can be found easily from this formula.
\begin{lemma}\cite[Theorem 4.3]{Panja2023roots}
        Let $a_n$ denote the number of elements in $\GL_n(q)$ which are $M$-th root of identity, where $M\geq 2$, $(M,q)]=1$, $q=p^a$ for some $a$. Let 
    $M=t\cdot p^r$, where $p\nmid t$. Then the generating function of the probability $a_n/|\GL_n(q)|$ is given by
    \begin{align*}
    &1+\sum\limits_{n=1}^{\infty}\dfrac{a_n}{|\GL_n(q)|}z^n=
        \prod\limits_{d|M}\left(1+\sum\limits_{m\geq 1}\sum\limits_{\substack{\lambda\vdash m\\\lambda_1\leq p^r}}
    \dfrac{z^{me(d)}}{q^{e(d)\cdot(\sum_{i}(\lambda_i')^2)}\prod\limits_{i\geq 1}\left(\dfrac{1}{q^{e(d)}}\right)_{m_i(\lambda_{\varphi})}}\right)^{\frac{\phi(d)}{e(d)}},
    \end{align*}
    where $e(d)$ denotes the multiplicative order of $q$ in $\Z/d\Z^\times$.
\end{lemma}

\section{Number of square roots in some finite groups of Lie type}\label{sec:Lie-type}
In this section, we will identify the elements of $G_n(q)$ that possess a square root within $G_n(q)$, where $G_n(q)$ represents one of the finite general, unitary, symplectic, or orthogonal groups.
This is accomplished by analyzing the conjugacy classes of $G_n(q)$, as any element with a square root will have all of its conjugates having square roots as well.
By $G_0(q)$ we mean the trivial group and $\C_{G_n}(\alpha)$ will be used to denote the centralizer of an element $\alpha$ in $G_n(q)$.
We start with a corollary to \cref{lem:factorization-of-power}.
\begin{corollary}\label{cor:factorization-square}
    Let $f$ be an irreducible polynomial of degree $d$ over a finite field
    $\fq$ of odd order and all irreducible factors of $f(x^2)$ are of degree greater than or equal to $d$.
    Then either $f(x^2)$ is irreducible or $f(x^2)$ has exactly two
    distinct irreducible factors of degree $d$. If $f(x^2)$ is irreducible the matrix $\diag(C_f,C_f)$ has a square root and otherwise  
    $C_f$ has a square root.
\end{corollary}
A monic irreducible polynomial $f$ is called a \emph{$2$-power} polynomial if $f(x^2)$ has a monic irreducible factor of degree as $\deg f$. 
The set of all $2$-power polynomials will be denoted as $\Phi_2$. 
If $f$ is a monic irreducible polynomial such that $f(x^2)$ is irreducible, we call it to be a \emph{skew $2$-power} polynomial. 
The set of all monic irreducible skew $2$-power polynomials will be denoted by $\Phi_2'$.
We note down some easy observations.
Let $\mathscr{C}$ be a conjugacy class of $\G_n$ and $g_1\sim g_2$ be two elements of $\mathscr C$. Then if $g_1$ has $m$ square roots, so does $g_2$. 
Thus it is enough to find the number of square roots for a particular element in $\mathscr C$. 
It is well known that the centralizers of two conjugate elements are conjugate to each other and thus have the same number of elements. 
Let $g\in\G_n(q)$ be a solution of the equation $X^2=\alpha$. Then $\C_{\G_n}(\alpha)$-conjugates of $g$ are also solutions of $X^2=\alpha$.  
If $g_1$ and $g_2$ are conjugate to each other and both are solutions of $X^2=\alpha$, it follows that they are conjugate by an element of $\C_{\G_n}(\alpha)$.
Because of the containment $\C_{\G_n}(g)\subseteq \C_{\G_n}(\alpha)$, these conjugates of $g$ amount to $[\C_{\G_n}(\alpha):\C_{\G_n}(g)]$ terms.
\begin{lemma}\cite[Theorem 2.3.14]{Franceschi2020}\label{lem:conj-cent-GL-n}
    Let $x\in \GL_n(q)$ has combinatorial data $(f_i,\lambda_{f_i})_{i=1}^{h}$
    with $\lambda_{f_i}=\lambda_{i,1}^{l_i,1}\lambda_{i,2}^{l_i,2}\cdots \lambda_{i,k_i}^{l_i,k_i}$ and $\deg f_i=d_i$. Then $$|\C_{\GL_n(q)}(x)|=q^{\gamma}\cdot 
    \left|\prod\limits_{i=1}^{h}\left(\prod\limits_{j=1}^{k_i}\GL_{l_{i,j}}(q^{d_i})\right)\right|,$$ where $\gamma=\sum\limits_{i=1}^{h}d_i\left(2\sum\limits_{u<v}\lambda_{i,u}l_{i,u}l_{i,v}+\sum\limits_{j}(\lambda_{i,j}-1)l_{i,j}^2\right)$.
\end{lemma}

\begin{corollary}\label{cor:squareroot-general}(See also \cite[Proposition 4.6]{KunduSingh22})
   Let $\alpha\in \GL_n(q)$ be a matrix with combinatorial data consisting of polynomials $f$ and partitions $\lambda_f=1^{m_1(\lambda_f)}2^{m_2(\lambda_f)}\cdots$. Then $\alpha$ has a square root if and only if for all $f$,
   \begin{enumerate}
       \item either $f$ is a $2$-power polynomial
       \item or $f(x^2)$ is irreducible and $2|m_j(\lambda_f)$ for all $j$.
   \end{enumerate}
\end{corollary}
To determine the number of square roots of a arbitrary element, we first start with the Jordan blocks of a single polynomial. 
We divide this into two cases; first when $f$ is a skew $2$-power polynomial and next when $f$ is a $2$-power polynomial. 
When $f$ is a skew $2$-power polynomial, $F(x)=f(x^2)$ is an irreducible polynomial and $J_{F,b}^2$ is conjugate to $\diag(J_{f,b}, J_{f,b})$.
Note that any square root of $\diag(J_{f,b}, J_{f,b})$ is conjugate to $J_{F,b}$ (consider them as elements in $\GL_{b\cdot\deg F}(q^{\deg F})$). 
Hence in this case any matrix with Jordan data $(f,\lambda_f=(1^{m_1}2^{m_2}\cdots k^{m_k}))$ has 
$
|\C_{\GL_{n}(q)}(J_{f,\lambda_f})|/|\C_{\GL_n(q)}(J_{F,\lambda_F})|
$
square roots, where $\lambda_F$ is given by $1^{m_1/2}2^{m_2/2}\cdots k^{m_k/2}$. Using \cref{lem:conj-cent-GL-n}, this number is
$$
q^{\gamma(\lambda_f)}\dfrac{\prod\limits_{j=1}^{k}|\GL_{m_j}(q^{\deg f})|}{\prod\limits_{j=1}^{k}|\GL_{m_j/2}(q^{2\deg f})|},
$$
where $\gamma(\lambda_f)=\sum\limits_{1\leq u<v\leq k}\frac{3um_um_v}{4}+\sum\limits_{u=2}^{k}\frac{3(u-1)m_u^2}{4}$. 
Now we consider the full Jordan block of a $2$-power polynomial.
Note that this matrix has two different square roots up to conjugacy, as $f(x^2)$ has exactly two distinct irreducible monic factors of degree $\deg f$, say $F_1$ and $F_2$. 
In this case $J_{F_i,b}^2$ is conjugate to $J_{f,b}$ for $i=1,2$.
Unlike the previous case, not any two square roots of $J_{f,b}$ are conjugate to each other. 
However, they are of the same type (see \cite[pp. 407]{Green55Characters}). 
Even $J_{F_i,b}$ and $J_{f,b}$ are of the same type, implying that their centralizer has the same number of elements.
This implies that any matrix with Jordan data $(f,\lambda_f=(1^{m_1}2^{m_2}\cdots k^{m_k}))$ has $$2^{\ell(\lambda_f)}-1$$ many square roots, where $\ell(\lambda_f)$ is number of $j$'s such that $m_j\neq 0$. 
Since for an arbitrary element with combinatorial data $\{(f,\lambda_f):f\in\Phi\}$ the size of the centralizer is the product of the size of the centralizers of the individual Jordan blocks, the following proposition follows from the above discussion.
\begin{proposition}\label{pro:number-of-square-root-arbitrary-GL}
    Let $\alpha\in\GL_n(q)$ has a square root in $\GL_n(q)$.
    Let the combinatorial data attached to $\alpha$ be given by
    \begin{align*}
        &\left\{(f_i,(1^{m_1(\lambda_{f_i})},2^{m_2(\lambda_{f_i})},\ldots)):1\leq i\leq r, f_i\in\Phi_2\right\}\\
        \bigcup&\left\{(g_j,(1^{m_1(\lambda_{g_j})},2^{m_2(\lambda_{g_j})},\ldots)):1\leq j\leq s, g_j\in\Phi_2'\right\}.
    \end{align*}
    Then the number of square roots of $\alpha$ is given by
    \begin{align*}
        \R_{\GL}(\alpha,2)=\prod\limits_{i=1}^{r}\left(q^{\gamma(\lambda_{f_i})}\dfrac{\prod\limits_{j=1}^{k_i}|\GL_{m_j(\lambda_{f_i})}(q^{\deg f_i})|}{\prod\limits_{j=1}^{k_i}|\GL_{m_j(\lambda_{f_i})/2}(q^{2\deg f_i})|}\right)\cdot
        \prod\limits_{j=1}^{s}\left(2^{\ell(\lambda_{g_j})}-1\right),
    \end{align*}
    where $\gamma(\lambda_{f_i})=\sum\limits_{1\leq u<v\leq k_i}\frac{3um_u(\lambda_{f_i})m_v(\lambda_{f_i})}{4}+\sum\limits_{u=2}^{k}\frac{3(u-1)m_u(\lambda_{f_i})^2}{4}$ and $\ell(\lambda_{g_j})$ is the number of $t$'s such that $m_t(\lambda_{g_j})\neq 0$.
\end{proposition}
We document the outcomes analogous to \cref{cor:squareroot-general} specifically for finite unitary, symplectic, and orthogonal groups. The proofs of these results involve essential yet straightforward adjustments, which we opt to exclude here. They represent appropriately adapted versions of the results presented in \cite{PanjaSingh22symporth} and \cite{PanjaSingh2023unitary}.
We present a few necessary definitions now. A \emph{self-reciprocal polynomial} of degree $d$ is a polynomial $f$ satisfying $f(0)\neq 0$ and $f(x)=f^{*}(x)=f(0)^{-1}x^df(x^{-1})$. 
A self-reciprocal irreducible monic polynomial $f$ of degree $r$ is called a \emph{$2^*$-power} polynomial if it has an irreducible self-reciprocal monic factor of degree $r$.
If $f$ is a self-reciprocal irreducible monic and $f(x^2)$ is irreducible, call $f$ to be a \emph{skew $2^*$ power polynomial}.
The set of all $2^*$-power (resp. skew $2^*$-power) polynomials will be denoted by ${\Phi_2^*}$ (resp. ${\Phi_2^*}'$). 
Consider the field $\mathbb F_{q^2}$, and the map $\sigma \colon \mathbb F_{q^2}\longrightarrow\mathbb F_{q^2}$ defined as $\sigma(a):=\overline{a}=a^q$ which is an order two automorphism of the field. 
This further induces an automorphism of $\mathbb F_{q^2}[t]$, the polynomial ring over $\mathbb F_{q^2}$. 
The image of $f\in\mathbb F_{q^2}[t]$ will be denoted by $\overline{f}$. 
A polynomial $f$ of degree $d$ is called \emph{self-conjugate} if $\widetilde{f}=f$ where $\widetilde{f}(t) = \overline{f(0)}^{-1} t^d \bar f(t^{-1})$. 
A self-conjugate irreducible monic polynomial $f$ of degree $d$ will be called an \emph{$\widetilde{2}$-power polynomial} if $f(x^2)$ has a self-conjugate irreducible monic factor $g$ of degree $d$. 
The set of all irreducible $\widetilde{2}$-power polynomials will be denoted as $\widetilde{\Phi}_2$.
A self-conjugate monic irreducible polynomial $f$ will be called a \emph{skew $\widetilde{2}$-power polynomial} is $f(x^2)$ is irreducible.
The set of all irreducible $\widetilde{2}$-power (skew $\widetilde{2}$-power) polynomials will be denoted as $\widetilde{\Phi}_2$ (resp. ${\widetilde{\Phi}_2}'$). 
The conjugacy classes of finite unitary groups, finite symplectic and orthogonal groups can be found in \cite{Ennola62} and \cite{Wall63}.
The determination of semisimple classes for which square-root exists, has been carried out for symplectic and orthogonal groups in \cite{PanjaSingh22symporth} and for unitary groups in \cite{PanjaSingh2023unitary}. 

For the unitary groups we know that two elements in $\U_n(q^2)$ are conjugate if and only if they are $\GL_n(q^2)$-conjugate. 
Also for each possible Jordan form of unipotent matrices, there is exactly one conjugacy class of it in $\U_n(q^2)$ (see \cite[Proposition 5.1.1]{Franceschi2020}). 
Combining all of these together we have the following
\begin{proposition}\label{prop:sqroot-gen-unitary}
       Let $\alpha\in \U_n(q^2)$ be a matrix with combinatorial data consisting of polynomials $f$ and partitions $\lambda_f=1^{m_1(\lambda_f)}2^{m_2(\lambda_f)}\cdots$. Then $\alpha$ has a square root in $\U_n(q^2)$ if and only if for all $f$ one of the following holds
   \begin{enumerate}
       \item $f=\widetilde{f}$ and $f\in\widetilde{\Phi}_2$,
       \item $f=\widetilde{f}$, $f\in\widetilde{\Phi}_2'$ and $2|m_j(\lambda_f)$ for all $j$,
       \item $f\neq\widetilde{f}$ and $f\in\Phi_2$ (and $\widetilde{f}\in\Phi_2$),
       \item $f\neq \widetilde{f}$, $f\in \Phi_2'$ (and $\widetilde{f}\in\Phi_2'$) and $2|m_j(\lambda_f)$ (and $2|m_j(\lambda_{\widetilde{f}})$) for all $j$. 
   \end{enumerate}
\end{proposition}
A unipotent element of $\GL_{2n}(q)$ belongs to $\Sp_{2n}(q)$ precisely when all Jordan block of odd dimension occur with even multiplicity. 
There are $2^r$ many $\Sp_{2n}(q)$-conjugacy class for some $r$ (see \cite[Proposition 5.1.3]{Franceschi2020}) with same Jordan data, each corresponding to a \emph{type}. 
Although they split into several {types}, using the tensor product construction (see \cite[Lemma 3.4.7]{BurnessGiudici16Derangements}) it can be shown that under the squaring map the image is conjugate to the same type (see \cite[Corollary 7.6]{PanjaSingh22symporth}). 
To describe the conjugacy classes of matrices of $\Sp_{2n}(q)$ with eigenvalues $\pm1$, we need the concept of signed partition.
A \emph{symplectic signed partition} is a partition of a number $k$, such that the odd parts have even multiplicity and even parts have a sign associated with it. 
They actually correspond to the splitting of conjugacy classes of $\GL_{2n}(q)$ inside $\Sp_{2n}(q)$. We note that the combinatorial data corresponding to a conjugacy class of an element in $\Sp_{2n}(q)$ has the form 
\begin{align*}
&\left\{(f,\lambda_f=1^{m_1(\lambda_f)},2^{m_2(\lambda_f)},\ldots):f\in\Phi, f\neq x, x\pm 1\right\}\\
\bigcup&\left\{(h,\lambda_h=1^{\pm m_1(\lambda_h)},2^{m_2(\lambda_h)},3^{\pm m_3(\lambda_h)},\ldots), h=x\pm 1\right\}.    
\end{align*}
Noting that all the unipotent conjugacy classes have a square root in $\Sp_{2n}(q)$, we have the following result concerning the elements which has a square root in $\Sp_{2n}(q)$, in view of \cite[Corollary 7.4]{PanjaSingh22symporth}.
\begin{proposition}\label{prop:sqroot-gen-symplectic}
       Let $\alpha\in \Sp_{2n}(q)$ be a matrix with combinatorial data consisting of polynomials $f$ and partitions $\lambda_f$. Then $\alpha$ has a square root in $\Sp_{2n}(q)$ if and only if for all $f\neq x\pm 1$ one of the following holds
   \begin{enumerate}
       \item $f=f^*$ and $f\in{\Phi^*}_2$,
       \item $f=f^*$, $f\in{\Phi^*}_2'$ and $2|m_j(\lambda_f)$ for all $j$,
       \item $f\neq{f^*}$ and $f\in\Phi_2$ (and ${f^*}\in\Phi_2$),
       \item $f\neq {f^*}$, $f\in \Phi_2'$ (and ${f^*}\in\Phi_2'$) and $2|m_j(\lambda_f)$ (and $2|m_j(\lambda_{{f^*}})$) for all $j$. 
   \end{enumerate}
   Furthermore, all unipotent element has a square root and no element with combinatorial data $(x+1,\lambda_{x+1})$ has a square root.
\end{proposition}
A similar result holds for the orthogonal groups as well; however, we opt not to present the result here.
\section{Number of real conjugacy classes}\label{sec:app}
Recall that for a group $G$, a conjugacy class $\mathscr{C}$ is called real if $\mathscr{C}=\mathscr{C}^{-1}=\{g^{-1}:g\in\mathscr{C}\}$. 
It is a well-established fact that the count of real irreducible characters aligns with the tally of real conjugacy classes.
Brauer’s Problem $14$ asks whether the number of characters which arise from real irreducible representations (i.e. irreducible characters with Frobenius-Schur indicator $1$) can be described group theoretically.
Murray and Sambale have affirmed this in their work \cite[Theorem A]{MurraySambale2023Brauer}. 
In particular they have shown that for a finite group $G$, the number of real conjugacy classes is given by $|s(2)|/|G|$, where $s(2)=\{(g,h)\in G^2: g^2h^2=1\}$.
We apply the results of the previous section to find the cardinality of $s(2)$ in the case of $\GL_n(q)$, thereby finding the number of real conjugacy classes of $\GL_n(q)$. 
The same technique can be used to find the number of real conjugacy classes of the unitary group $\U_n(q)$, the symplectic group $\Sp_{2n}(q)$, and the orthogonal groups $\On{\epsilon}_n(q)$ for $\epsilon\in \{0,\pm\}$. 
We refrain from mentioning them. Here is the result for $\GL_n(q)$.
\begin{theorem}\label{thm:number-of-real-irreducible-characters}
    The number of real conjugacy classes of $\GL_n(q)$ is given by
    \begin{align*}
        1+\dfrac{1}{|\GL_n(q)|}\left(c_4+{c_2(c_2-1)}+\sum\limits_{\substack{\alpha\in\GL_n(q)^2\\\alpha\neq1,\,{\alpha^2\neq 1}}}{\R_{\GL}(\alpha,2)(\R_{\GL}(\alpha,2)-1)}\right),
    \end{align*}
    where $c_t$ is the number of elements of order $t$ for $t=2,4$, 
\end{theorem}
\begin{proof}
    We count the number of elements in $s(2)$. Note that $s(2)$ is a disjoint union of the sets $\{(g,g^{-1}):g\in \GL_n(q)\}$, $\{(g,g):o(g)=4\}$, $\{(g,h):g\neq h, g^2=h^2=1\}$ and $\{(g,h):g^2h^2=1, g^2\neq 1\neq h^2, g^2\neq g^{-2}\}$. 
    The last set has cardinality as same as $\{(g,h)|g^2=h^2\neq 1,g^2\neq g^{-2}\}$. This finishes the proof.
\end{proof}

\printbibliography
\end{document}